\documentclass[12pt]{amsart}

\usepackage{amsfonts}
\usepackage{amssymb}
\usepackage{amsmath}

\theoremstyle{plain}
\newtheorem{theorem}{Theorem}
\newtheorem{proposition}[theorem]{Proposition}

\newtheorem{lemma}[theorem]{Lemma}

\theoremstyle{definition}

\newcommand{\bC}{\mathbb{C}}

\newcommand{\cA}{\mathcal A}
\newcommand{\cB}{\mathcal B}

\begin{document}

\baselineskip 6.8mm

\title{A note on Huijsmans-de Pagter problem in ordered Banach algebras\footnote{To appear in \textit{Positivity}}}

\author{Roman Drnov\v sek}

\date{\today}

\begin{abstract}
We give an example of a positive element $a$ in some ordered Banach algebra $\cA$
such that its spectrum is equal to $\{1\}$ and it is not greater than or equal to the unit element of $\cA$. 
\end {abstract}

\maketitle

 {\it Math. Subj. Classification (2010)}: 47B60, 47A10, 46H05 \\

Let $\cA$ be a complex Banach algebra with unit $e$. A nonempty set $C$ is called a \textit{cone} of $\cA$ if 
$C + C \subseteq C$ and $\lambda C \subseteq C$ for all $\lambda \ge 0$. 
If, in addition, $C \cap (- C) = \{0\}$, then $C$ is said to be a \textit{proper} cone.
A cone $C$ of $\cA$ is \textit{closed} if it is a closed subset of $\cA$.
Any proper cone $C$ induces an ordering $\le$ in the following way: $a \le b \iff b - a \in C$.
It is easy to see that this ordering is a partial order (reflexive, antisymmetric, and transitive).
Clearly, $C = \{ a \in \cA: a \ge 0 \}$, and so elements of $C$ are called \textit{positive}.
A proper cone $C$ of $\cA$ is \textit{normal} if there exists a constant $\alpha > 0$ such that it follows from $0 \le a \le b$
that $\|a\| \le \alpha \|b\|$. 
A cone $C$ is called an \textit{algebra cone} of $\cA$ if $C \cdot C \subseteq C$ and $e \in C$.
In this case $\cA$ is called an \textit{ordered Banach algebra}.  For a general survey on ordered Banach algebras we refer to 
a recent paper \cite{MR17} and the references given there.

As it is mentioned in \cite{Zh93}, Huijsmans and de Pagter posed the following question: 
If $T$ is a positive operator on a complex Banach lattice $E$ with the spectrum $\sigma(T) = \{1\}$, does
it follow that $T$ is greater than or equal to the identity operator $I$?  
This question was studied in the papers \cite{Zh92}, \cite{Zh93} and \cite{Dr07}. 
In particular, Zhang \cite{Zh93} showed that the answer is affirmative 
under the additional hypothesis that there exist $\alpha \in (0, \frac{1}{2})$ 
and a constant $c \ge 0$ such that $\|T^{-n}\|=O(\exp (cn^\alpha ))$ as $n\to \infty.$ 
On the other hand, it was shown in \cite{Dr07} that  a positive operator $T$ on $E$ 
is greater than or equal to the identity operator $I$ provided 
$\lim_{n \rightarrow \infty} n \, \|(T - I)^n\|^{1/n} = 0$. This assumption clearly implies that $\sigma(T) = \{1\}$.
In general this problem that is important for the spectral theory of positive operators is still open. 

In \cite{Mo03} this question was investigated in the context of ordered Banach algebras,
that is, the paper gave some conditions under which a positive element $a$ in an  ordered Banach algebra
with $\sigma(a) = \{1\}$ is necessarily greater than or equal to the unit element.
However, \cite{Mo03} did not provide any example of a positive element $a$ in some ordered Banach algebra
such that $\sigma(a) = \{1\}$ and $a$ is not greater than or equal to the unit element.
In the present note we give such an example.
 
Let $\cB$ be a complex unital Banach algebra with unit $e$, and let $\cA$ be the algebra $\cB \times \bC$ 
endowed with multiplication  $(a, \xi) \cdot (b, \eta) = (a b, \xi \eta)$.
If we define the norm on the algebra $\cA$ by $\| (a, \xi) \| = \max\{ \|a\|, |\xi| \}$, then 
$\cA$ becomes a unital complex Banach algebra with the unit $(e,1)$.
Observe that $\sigma((a, \xi)) =  \sigma(a) \cup \{\xi\}$ for all $a \in \cB$ and $\xi \in \bC$.
Furthermore,  if $\cB$ is a $C^*$-algebra, then $\cA$ is also a $C^*$-algebra with the involution defined by
$(a, \xi)^* = (a^*, \overline{\xi})$. 

\begin{lemma} 
\label{lem} 
The Banach algebra $\cA$ is an ordered Banach algebra with the algebra cone
$$ K = \{ (a, \xi)  \in \cA : \|a\| \le \xi \}  $$
that is proper, closed and normal. 
Furthermore, if $e \neq a \in \cB$, $\| a \| = 1$ and $\sigma(a) = \{1\}$, then $(a, 1) \in K$, 
 $\sigma((a,1)) = \{1\}$ and $(a, 1) - (e,1) \not\in K$.
\end{lemma}

\begin{proof}
It is easy to see that $K$ is an algebra cone that is proper and closed.   
To show its normality, assume $0 \le (a, \xi) \le (b, \eta)$, so that 
 $\|a\| \le \xi$ and $\| b- a\| \le \eta-\xi$. Then $\|a\| \le \xi \le \eta$, and so 
$\| (a, \xi) \| \le \eta \le \| (b, \eta)\|$. This shows that $K$ is a normal cone.

Assume that $e \neq a \in \cB$, $\| a \| = 1$ and $\sigma(a) = \{1\}$.
Then $(a, 1) \in K$,  $\sigma((a,1)) = \sigma(a) \cup \{1\}= \{1\}$ and 
$(a, 1) - (e,1) = (a-e, 0) \not\in K$, since $\|a-e\| > 0$.
\end{proof}

The ordered Banach algebra  $\cA$ was used in \cite{HS07} to prove some theorems for elements of $\cB$ 
by working in $\cA$. We also mention that $K$ is an ice-cream cone in the normed space $\cA$ 
(according to the definition in \cite{AT07}), as $(a, \xi) \in K$ iff   $\| (a, \xi) \| \le \xi$. 

We now prove the main result of this note.

\begin{theorem}
\label{thm}
There exist an ordered Banach algebra $\cA$ with a closed and normal algebra cone 
and a positive element $a \in \cA$ such that 
$\sigma(a) = \{1\}$ and $a$ is not greater than or equal to the unit element of $\cA$.
\end{theorem}

\begin{proof}
Let $\cB$ be the Banach algebra of all bounded linear operators on the Hilbert space $L^2[0,1]$, 
and let $\cA = \cB \times \bC$ be an ordered Banach algebra as defined above. 
Let $V$ be the Volterra operator on $L^2[0,1]$, that is, the operator defined by 
$$ (V f)(x) = \int_0^x \! f(y) \, dy \ \ \ (f \in L^2[0,1] , \ x \in [0,1]) . $$
Since $\sigma(V) = \{0\}$, the operator $T = (I + V)^{-1}$ has the spectrum $\sigma(T) = \{1\}$, and it is not equal to the identity operator $I$.  
It is shown in \cite[Problem 190]{Ha82} that $\|T\|= 1$.  
Then it follows from Lemma \ref{lem} that $(T,1) \in  K$,  $\sigma((T,1)) = \{1\}$ and $(T, 1) - (I,1) \not\in K$.
\end{proof}

Infinite-dimensionality of the Hilbert space $L^2[0,1]$ is essential in the proof of Theorem \ref{thm}, as we have 
the following observation.

\begin{proposition}
Let $\cB$ be the Banach algebra of all linear operators on a finite-dimensional Hilbert space,
and let $\cA = \cB \times \bC$ be an ordered Banach algebra as defined above. 
If $(A, \xi) \in K$ with $\sigma((A, \xi)) = \{1\}$, then $(A, \xi)$ is equal to the unit element $(I,1)$ of $\cA$.
\end{proposition}

\begin{proof}
Since $\sigma((A, \xi)) = \{1\}$ and $\|A\| \le \xi$, we have  $\sigma(A) = \{1\}$, $\xi = 1$, 
$\|A\| = 1$, and $A$ is unitarily equivalent to triangular matrix 
that has only 1's on the diagonal. Its norm can be $1$ only in the case when $A= I$.
\end{proof}

\vspace{3mm}
{\it Acknowledgment.} The author acknowledges the financial support from the Slovenian Research Agency  (research core funding No. P1-0222).

\vspace{2mm}

\baselineskip 6mm
\noindent
Roman Drnov\v sek \\
Department of Mathematics \\
Faculty of Mathematics and Physics \\
University of Ljubljana \\
Jadranska 19 \\
SI-1000 Ljubljana, Slovenia \\
e-mail : roman.drnovsek@fmf.uni-lj.si 

\end{document}